\documentclass{amsart}

\usepackage{hyperref} 
\usepackage{marginnote}
\usepackage{color}
\usepackage{pgf,tikz, tikzscale}
\usetikzlibrary{arrows, knots, external, positioning}
\newtheorem{theorem}{Theorem}[section]
\newtheorem{lemma}[theorem]{Lemma}

\theoremstyle{definition}
\newtheorem{definition}[theorem]{Definition}

\theoremstyle{remark}
\newtheorem{remark}[theorem]{Remark}

\numberwithin{equation}{section}

\begin{document}

\title{Knot surgery $4$-manifolds $E(n)_K$ without $1$- and $3$-handles}

\author{Ju A Lee}
\address{Research institute of Mathematics, 
Seoul National University, Seoul  08826, Republic of Korea}
\email{jualee@snu.ac.kr}

\author{Ki-Heon Yun}
\address{School of Mathematics, Statistics and Data science, 
Sungshin Women's University, Seoul 02844, Republic of Korea}
\email{kyun@sungshin.ac.kr}
\thanks{The first named author was supported by the National Research Foundation of Korea (NRF) Grant funded by the Korean government (RS-2024-00392067). The second named author was supported by the National Research 
Foundation of Korea (NRF) Grant funded by the Korean government (2019R1A2C1010302).}


\date{\today}



\begin{abstract}
In this article, we demonstrate  that for any positive integer $n$,  the knot surgery $4$-manifold $E(n)_K$ has a handle 
decomposition without $1$- and $3$-handles.  Here, $K$ represents either a fibered two-bridge knot $C(2\epsilon_1, 2\epsilon_2,\cdots, 2\epsilon_{2g})$ ($\epsilon_i \in \{ 1, -1\}$) in Conway's notation or a Stallings knot $K_m$ ($m \in \mathbb{Z}$).
\end{abstract}

\maketitle

\section{Introduction}

A long standing question in smooth $4$-manifold theory is whether a closed
simply connected $4$-manifold admits a handle decomposition 
without $1$-handles
or without $1$- and $3$-handles (Kirby's problem list 4.18). 
Significant progress has been made on this problem: 
Akbulut \cite{Akbulut:12, Akbulut:09} showed that the Dolgachev surface $E(1)_{2,3}$ has a handle decomposition without $1$- and $3$- handles. 
Baykur~\cite{Baykur:25} demonstrated the existence of  irreducible, closed, simply connected $4$-manifolds that admit a perfect Morse function. These manifolds have prescribed signature and spin type, provided the signature is divisible by $16$ and the Euler characteristic is sufficiently large.
Monden and Yabuguchi~\cite{Monden:25} proved that knot-surgered elliptic surfaces arising from the $(2, 2h+1)$-torus knot, which is a family of two-bridge knots, also admit handle decompositions without $1$- and $3$-handles. 

In this article, we extend these results to show that for any positive integer $n$ and fibered two-bridge knot $K$, the knot surgery $4$-manifold $E(n)_K$ has a handle decomposition without $1$- and $3$-handles. We further show that this property holds for any Stallings knot $K_m$, which is a $3$-bridge knot.


\section{Knot surgery \texorpdfstring{$4$}{4}-manifolds \texorpdfstring{$E(n)_K$}{E(n)K} and Lefschhetz fibration}

 \subsection{ Monodromy factorization of \texorpdfstring{$E(n)_K$}{E(n)K}}
 \label{subsection:MF}

Suppose 
 $X$ is a simply connected, smooth $4$-manifold containing an embedded torus $T$ of square $0$. 
 For any knot $K \subset S^3$,
 one can construct a new $4$-manifold,
 known as  a {\em  knot surgery $4$-manifold}, denoted $X_K$.
 This is achieved by performing a fiber sum
 \begin{equation*}
  X_K = X\sharp_{T=T_m} (S^1\times M_K)
 \end{equation*}
 along a torus $T$ in $X$ and $T_m = S^1\times m$ in
 $S^1 \times M_K$, where $M_K$ is the $3$-manifold obtained by
 $0$-framed surgery along $K$, and $m$ is the meridian of $K$.

 
 If $X$ is a simply connected elliptic surface $E(n)$,
 $T$ is the elliptic fiber, and $K$ is a fibered knot of genus $g(K)$,
 then the knot surgery $4$-manifold $E(n)_K$ admits both
 symplectic structure and a genus $2g(K)+n-1$ Lefschetz
 fibration structure~\cite{FS:2004, Yun:2008}.
 
  \begin{figure}[tbh]
 \begin{center}
\includegraphics[scale=0.3]{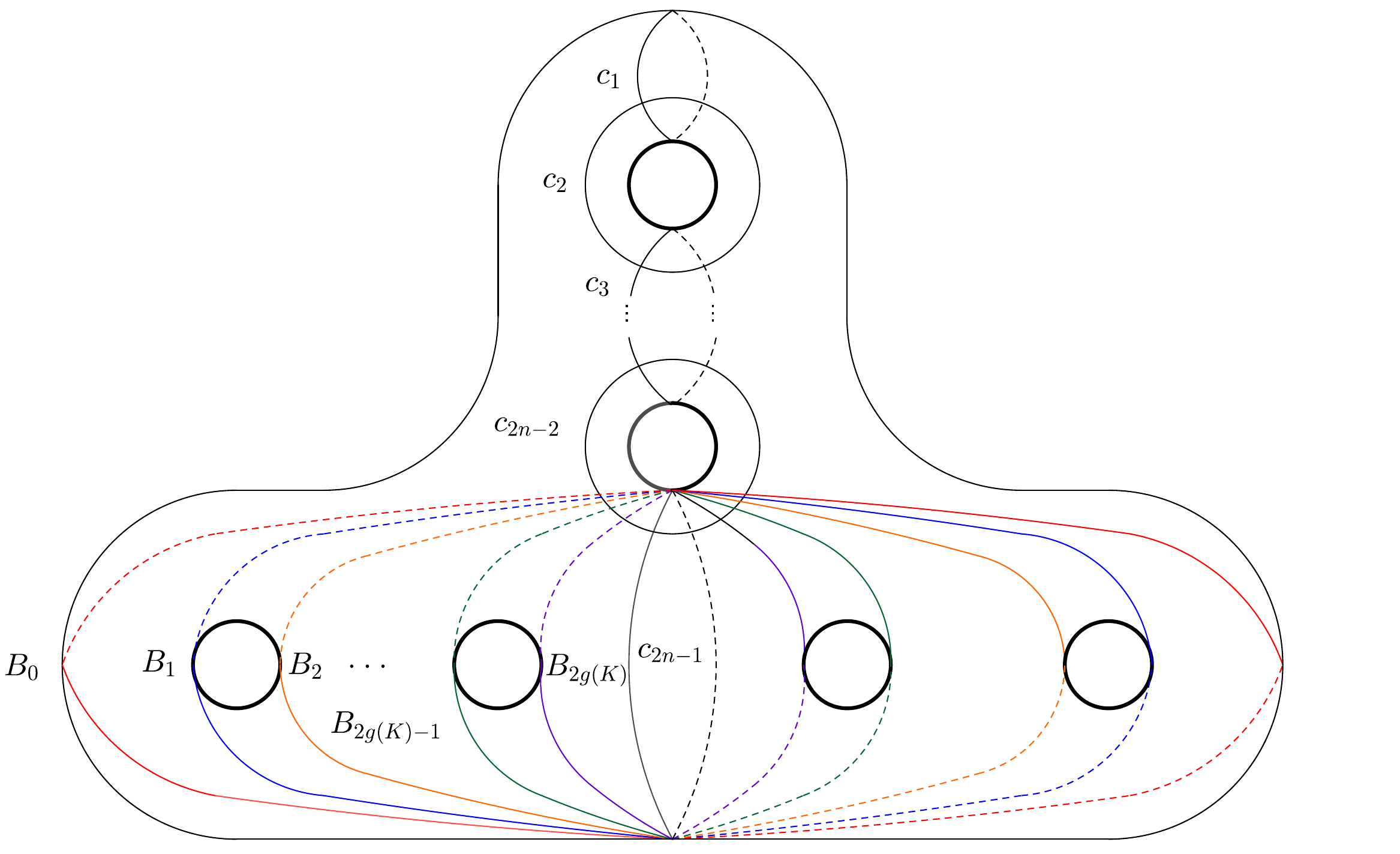}
\caption{Vanishing cycles on $E(n)_K$}
\end{center}
\end{figure}

A symplectic Lefschetz fibration over $S^2$ is characterized by its monodromy factorization, 
which is an ordered sequence of right-handed Dehn twists considered up to Hurwitz moves and simultaneous conjugation. It is known that $E(n)_K$ admits a Lefschetz fibration whose monodromy factorization is given by  
\[
\Phi_K( W^2 )\cdot W^2,
\]
where
$\phi_K : \Sigma_{g(K)}^1 \to  \Sigma_{g(K)}^1$ is the monodromy of the fibered knot $K$, defined by
\[
S^3 - \nu(K) =  ([0,1] \times  \Sigma_{g(K)}^1) / _{(1, x) \sim (0, \phi_K(x))}
\]
and
\begin{eqnarray*}
\Phi_K &=& \phi_K \oplus id \oplus id : \Sigma_{g(K)} \sharp \Sigma_{n-1} \sharp \Sigma_{g(K)}  \to \Sigma_{g(K)} \sharp \Sigma_{n-1} \sharp \Sigma_{g(K)},\\
W &=& t_{c_{2n-2}} \cdots  t_{c_2} t_{c_1} t_{c_1} t_{c_2} \cdots t_{c_{2n-2}} t_{B_0} t_{B_1}  \cdots t_{B_{2g(K)}} t_{c_{2n-1}}.
\end{eqnarray*}
The word $W$ arises from a hyperelliptic involution and  Matsumoto's word ~\cite{Matsumoto:96, Gurtas:2004, Yun:2006}.

Since $ W$ corresponds to an involution on $\Sigma_{2g(K) + n -1}$, we obtain the following equivalent relation:
\[
\Phi_K( W^2 )\cdot W^2  \sim \Phi_K( W )\cdot W \cdot W\cdot \Phi_K(W).
\]
 
Let $X_1$ be the Lefschetz fibration over the closed two-dimensional disk $D^2$ with  monodromy factorization $\Phi_K(W) \cdot W$, and let $X_2$  be the Lefschetz fibration over $D^2$ with  monodromy factorization $W \cdot \Phi_K(W)$. 
Then 
\[
E(n)_K = X_1 \cup_{\psi} X_2, \qquad \psi: \partial X_2 \to \partial X_1.
\] 
In this construction, $\psi$ is a fiber preserving and orientation reversing diffeomorphism; specifically, it is the identity map on fiber $\Sigma_{2g(K) + n -1}$ and the conjugation on $S^1 = \partial D^2$.


\subsection{ Handle decomposition of \texorpdfstring{$X_i$}{Xi} (\texorpdfstring{$i=1, 2$}{i=1, 2}) and \texorpdfstring{$E(n)_K$}{E(n)K}}\label{subsection:LF}

Since the surface $\Sigma_{2g(K) + n-1}$  possesses a $2$-dimensional handle decomposition consisting of one $0$-handle, $(4g(K) + 2n -2)$  $1$-handles, and one $2$-handle (as shown in the Figure~\ref{fig:Handle_fiber}), the product 
\[
\Sigma_{2g(K) + n -1} \times   D^2 
\]
admits a handle decomposition comprising one $4$-dimensional $0$-handle, $(4g(K) + 2n -2)$  $4$-dimensional $1$-handles, and one $4$-dimensional $2$-handle with framing $0$. 
Subsequently,  we attach $4(g(K) + 2n -1)$  $4$-dimensional $2$-handles, one for each right-handed Dehn twist in the monodromy factorization $\Phi_K(W) \cdot W$. 
These $2$-handles are attached sequentially, starting from the rightmost twist. 
For each Dehn twist $t_c$, the attaching circle $c$ is located on a disjoint fiber surface, and the framing is defined as one less than the surface framing. 

Consequently, each manifold $X_i$ (for $i=1, 2$) possesses a handle decomposition consisting of
\begin{itemize}
\item one $4$-dimensional $0$-handle, 
\item $(4g(K) + 2n -2)$ $4$-dimensional $1$-handles, and 
\item $(4g(K) + 8n -3)$ $4$-dimensional $2$-handles.
\end{itemize} 
Since $E(n)_K$ decomposes as $X_1 \cup_{\psi} X_2$, its handle decomposition is formed by taking the handle structure of $X_1$ and attaching the upside-down handlebody of $X_2$ via the boundary diffeomorphism.


\subsection{Fibered two-bridge knots and their monodromy}

Assume that $p$ and $q$ are relatively prime integers with $p$ being odd.
We consider a two-bridge knot $b(p,q)$ defined as follows:

\begin{figure}[htb]
\begin{center}
\includegraphics[scale=0.21]{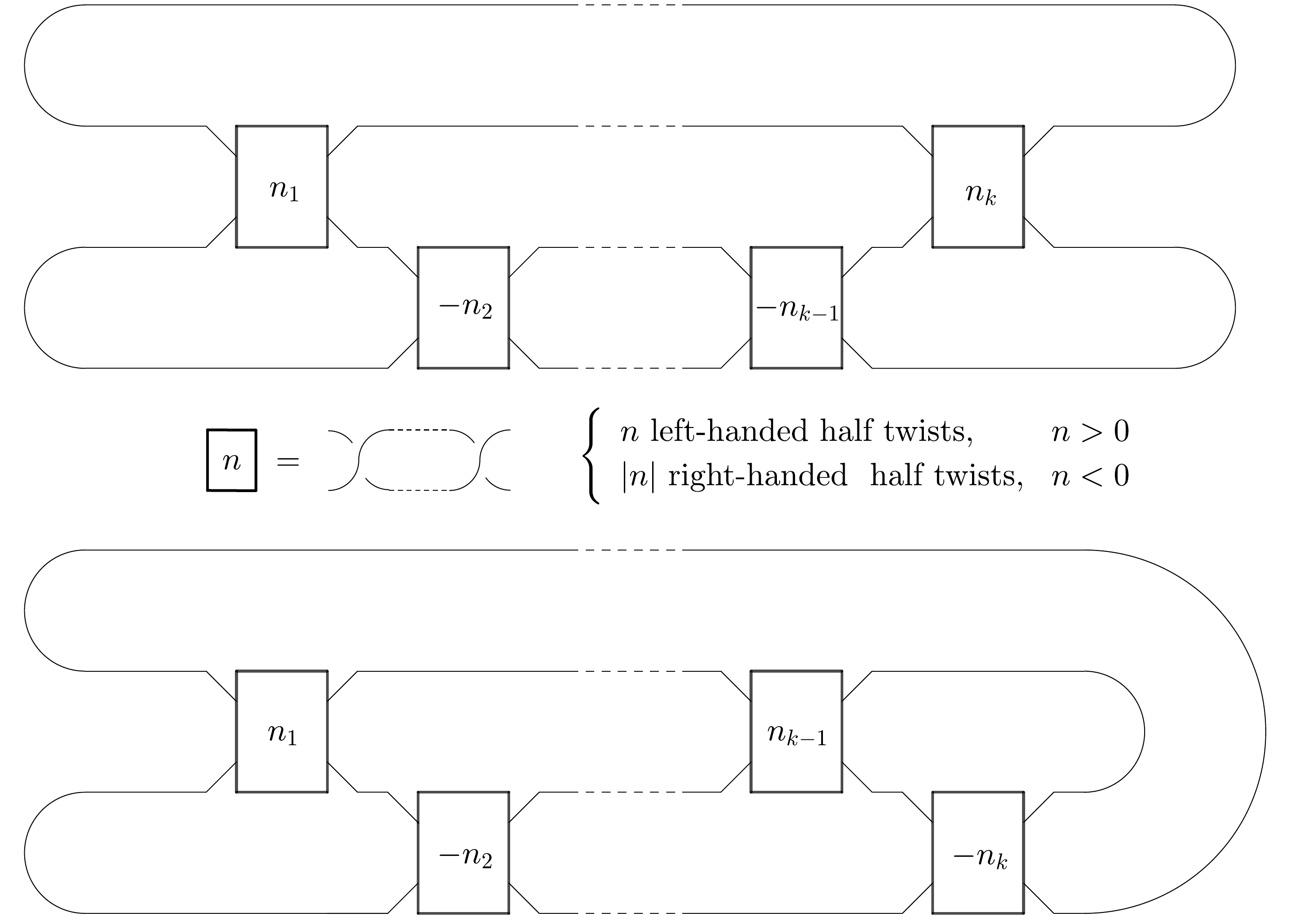}
\end{center}
\caption{A $2$-bridge knot  $C(n_1, n_2, \cdots,  n_k)$ (for $k$ odd and $k$ even, respectively)}\label{fig:2bridge}
\end{figure}

\begin{definition}[\cite{BZ:03}]
\label{Definition:2bridge}
A \emph{two-bridge link} $b(p,q)$ has the Conway's normal form
\[ C(n_1, n_2,  \cdots, n_k)\]
as in Figure~\ref{fig:2bridge}, where
\[
\frac{p}{q} = n_1+\frac{1}{n_2+\frac{1}{\ddots
\frac{1}{n_{k-1} + \frac{1}{n_k}}}} =[n_1, n_2, \cdots, n_k].
\]
\noindent
This knot is homotopic to a $4$-plat defined by the braid:
\[
\begin{cases}
\sigma_2^{n_1} \sigma_1^{-n_2} \sigma_2^{n_3} \sigma_1^{-n_4} \cdots
\sigma_1^{-n_k + \epsilon} \sigma_2^\epsilon, & {\textrm{if $k$ is even  ($\epsilon \in \{ +1, -1\}$) }}, \\
\sigma_2^{n_1} \sigma_1^{-n_2} \sigma_2^{n_3} \sigma_1^{-n_4} \cdots  \sigma_2^{n_k}, & {\textrm{if $k$ is\, odd}}.
\end{cases}
\]
Here,  $\sigma_i$ represents the standard braid generator.
 We denote $C(2m_1, -2m_2, \cdots, \allowbreak 2m_{2k-1}, -2m_{2k})$ by
$D(m_1, m_2, \cdots, m_{2k-1},  m_{2k})$.
\end{definition}

\smallskip

\begin{figure}[htb]
\label{fig:Curves}
\begin{center}
\includegraphics[scale=0.2]{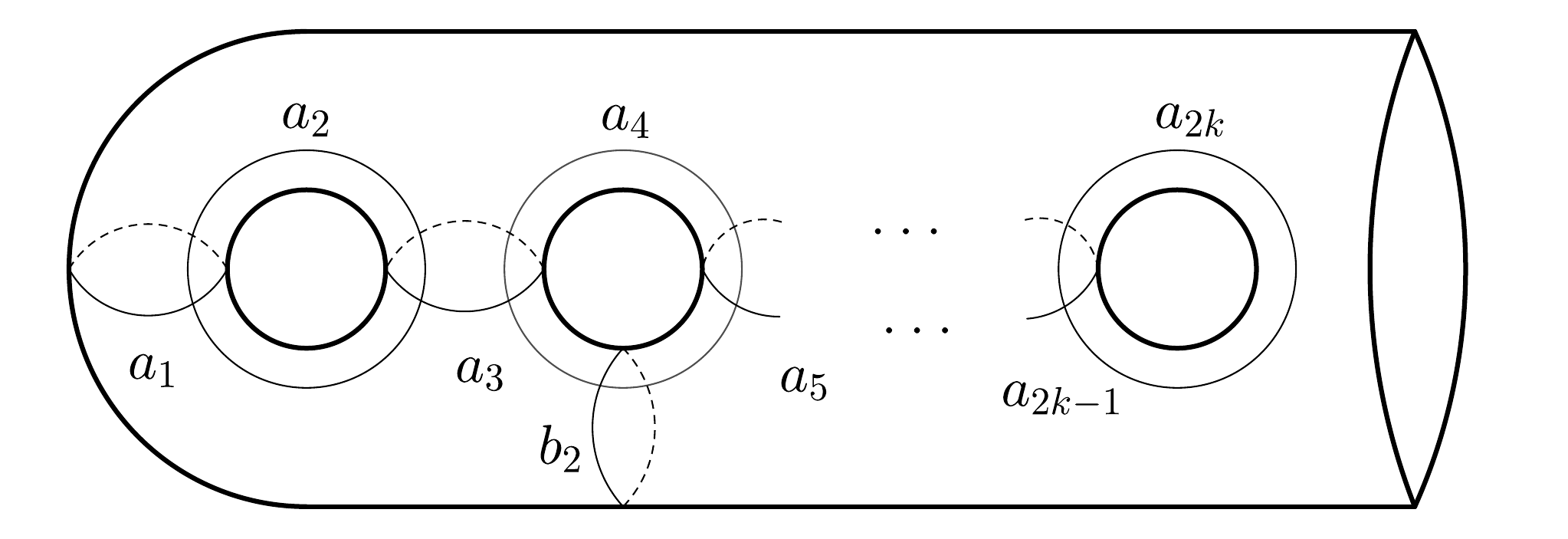}
\end{center}
\caption{Curves for $\phi_{D(\epsilon_1, \epsilon_2,\cdots, \epsilon_{2k})}$}\label{fig:2bridge_monodromy}
\end{figure}

\begin{lemma}[\cite{BZ:03, Kawauchi:1996}]
\label{Lemma:2-bridge}
\begin{itemize}
\item[(a)] Two two-bridge knots $b(p,q)$ and  $b(p',q')$ are equivalent if and only if
$p=p'$ and either $q=q'$ or $qq' \equiv 1  \pmod{p}$.
\item[(b)] The knot $D(n_1, n_2, \cdots, n_{2k})$ is a fibered knot if and only if
each $n_i$ is $1$ or $-1$.
\item[(c)] Two two-bridge knots $D(a_1, a_2, \cdots, a_{2k})$ and $D(b_1, b_2, \cdots, b_{2l})$
are ambient isotopic  if and only if $k=l$ and $a_i =b_i$ or $a_i =
b_{2k+1-i}$ for all $1\le i \le 2k$.
\item[(d)] The knot $K = D(\epsilon_1, \epsilon_2, \cdots, \epsilon_{2k})$, where each $\epsilon_i \in \{ +1, -1\}$, has the monodromy
\[
\phi_K=t_{a_{2k}}^{\epsilon_{2k}} \circ \cdots \circ t_{a_2}^{\epsilon_2} \circ t_{a_1}^{\epsilon_1}.
\]
\end{itemize}
\end{lemma}

\begin{proof}
The peoperties (a), (b) and (c) are well known facts from the classification of Lens spaces.
The Seifert surface of $D(\epsilon_1, \epsilon_2, \cdots, \epsilon_{2k})$ is obtained by plumbing a sequence of positive Hopf band for $\epsilon_i = +1$ and negative Hopf band for $\epsilon_i = -1$. Note that a positive Hopf band is the Seifert surface of a positive Hopf link.

Each positive Hopf band contributes a Dehn twist $t_c$, and negative Hopf band contributes a Dehn twist $t_c^{-1}$, where $c$ is the core circle of the corresponding Hopf band.
Therefore, we obtain the monodromy: 
\[
\phi_K=t_{a_{2k}}^{\epsilon_{2k}} \circ \cdots \circ t_{a_2}^{\epsilon_2} \circ t_{a_1}^{\epsilon_1}
\]
where $a_i$ is the simple closed curve on $\Sigma_{g(K)}^1$ as in Figure~\ref{fig:2bridge_monodromy}.
Recall that $\phi_K$ is defined by   
\[
S^3 - \nu(K) = ([0,1] \times \Sigma_{g(K)}^1)/ _{(1,x) \sim (0, \phi_K(x))}
\]
which  considers the direction opposite to Harer's result ~\cite{Harer:82}.
\end{proof}


\section{Fibered two-bridge knot \texorpdfstring{$K$}{K} and handle decomposition of \texorpdfstring{$E(n)_K$}{E(n)K}}

In this section, we identify $(4g(K) + 2n -2)$ $1$-handle/$2$-handle canceling pairs in the handle decomposition of $X_i$.  
To locate these canceling pairs, we first establish an algorithm to analyze the vanishing cycles in the monodromy factorization $\Phi_K(W) \cdot W$. Each vanishing cycle corresponds to  a curve on $\Sigma_{2g(K) + n-1}^1$, the surface obtained from a $2$-dimensional $0$-handle by attaching $(4g(K) + 2n -2)$  $2$-dimensional $1$-handles.

\begin{figure}[tbh]
\begin{tabular}{ll}
\includegraphics[scale=0.19]{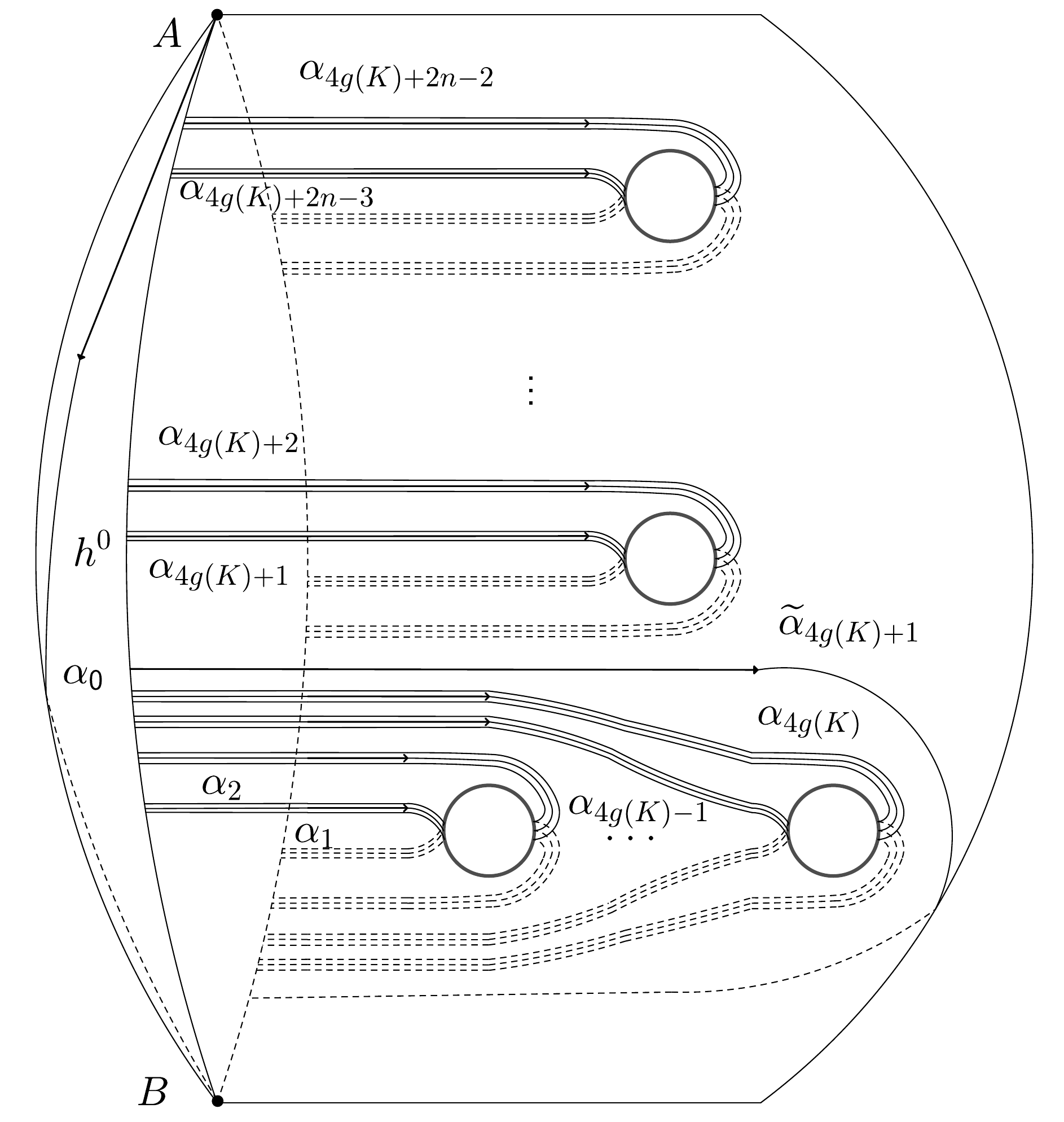} &
\includegraphics[scale=0.21]{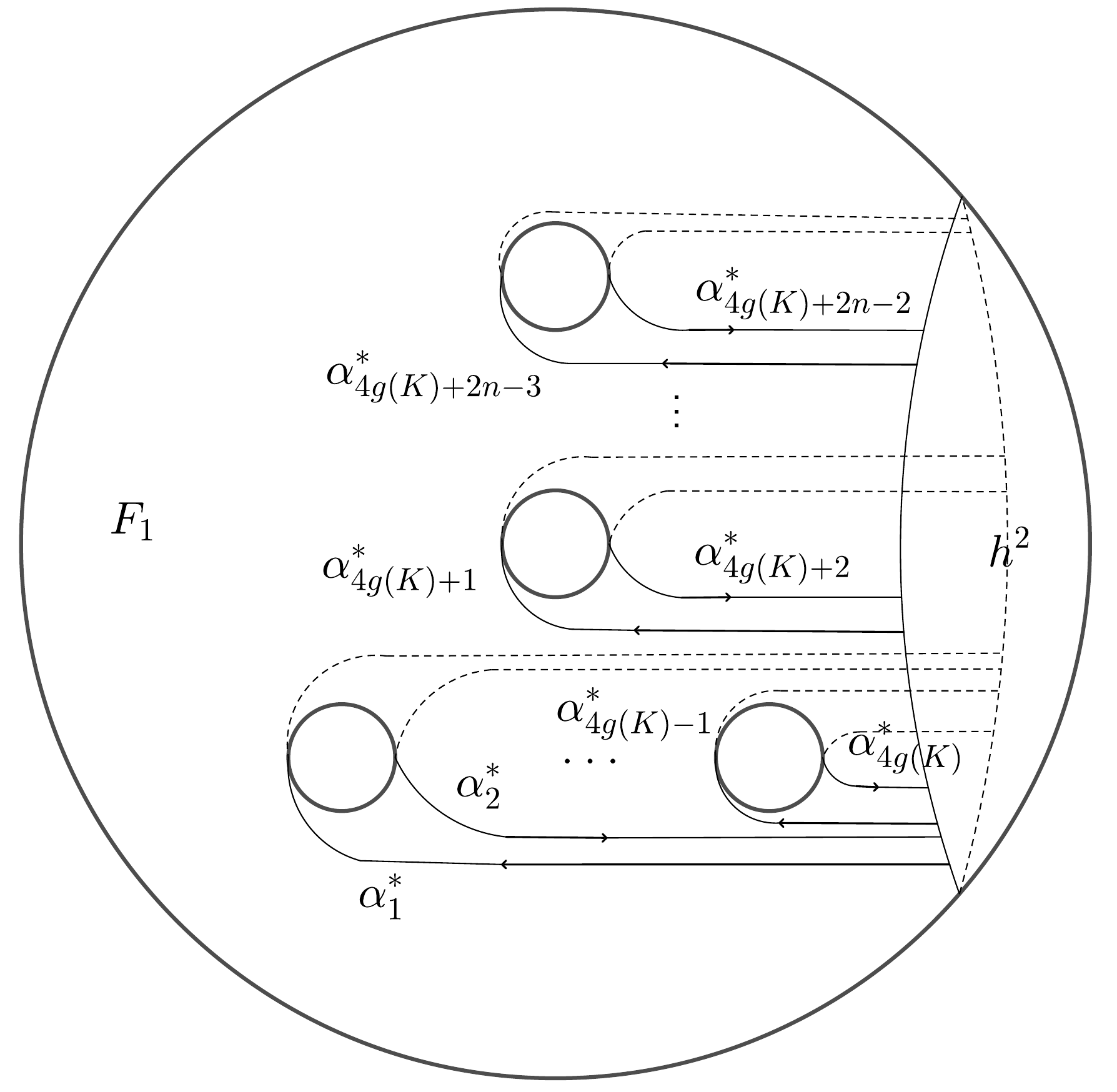}
\end{tabular}
\caption{Handle decomposition of  $\Sigma_{2g(K) + n-1}^1$: core $\alpha_i$ and co-core $\alpha_i^*$ of each $2$-dimensional $1$-handle (ordered set of oriented arcs  $\{ \alpha_i, \alpha_i^*\}$ give the positive orientation of the fiber surface.)}\label{fig:Handle_fiber}
\end{figure}

\begin{definition}
Let $F = \Sigma_{2g(K) + n-1}$ be the fixed fiber surface of Lefschetz fibration $X_i$ (subsection~\ref{subsection:LF}) and let $h^0$ be the $2$-dimensional $0$-handle, $h_i^1 = \alpha_i \times I$  be the $2$-dimensional $1$-handle with oriented core $\alpha_i$ and oriented co-core $\alpha_i^*$  for each $i = 1, 2, \cdots, 4g(K) + 2n -2$ as shown in Figure~\ref{fig:Handle_fiber}. 
We will use the same notation $\alpha_i^*$ for the $4$-dimensional $1$-handle $h_i^1 \times D^2 = \alpha_i \times (I \times D^2)$.
\end{definition}

\begin{definition}
Given three oriented paths $\alpha$ , $\beta$ and $\delta$ with end points on $h^0$, we define
\[ \delta \simeq \alpha * \beta\]
if the path $\alpha* \beta$ (obtained by identifying the terminal point of $\alpha$ and the starting point of $\beta$) is homotopic to $\delta$ relative to the boundary in $F_1 \times D^2$, where $F_1 = h^0\cup (\cup_{i=1}^{4g(K) + 2n -2} h_i^1)$. 
Collapsing the $4$-dimensional $0$-handle $h^0 \times D^2$ to a point $(b_0, \mathbb{O})$ implies $[\delta] = [\alpha]* [\beta]$ in $\pi_1(F_1\times D^2, (b_0, \mathbb{O}))$.
\end{definition}

\begin{figure}[tbh]
\begin{center}
\includegraphics[scale=0.25]{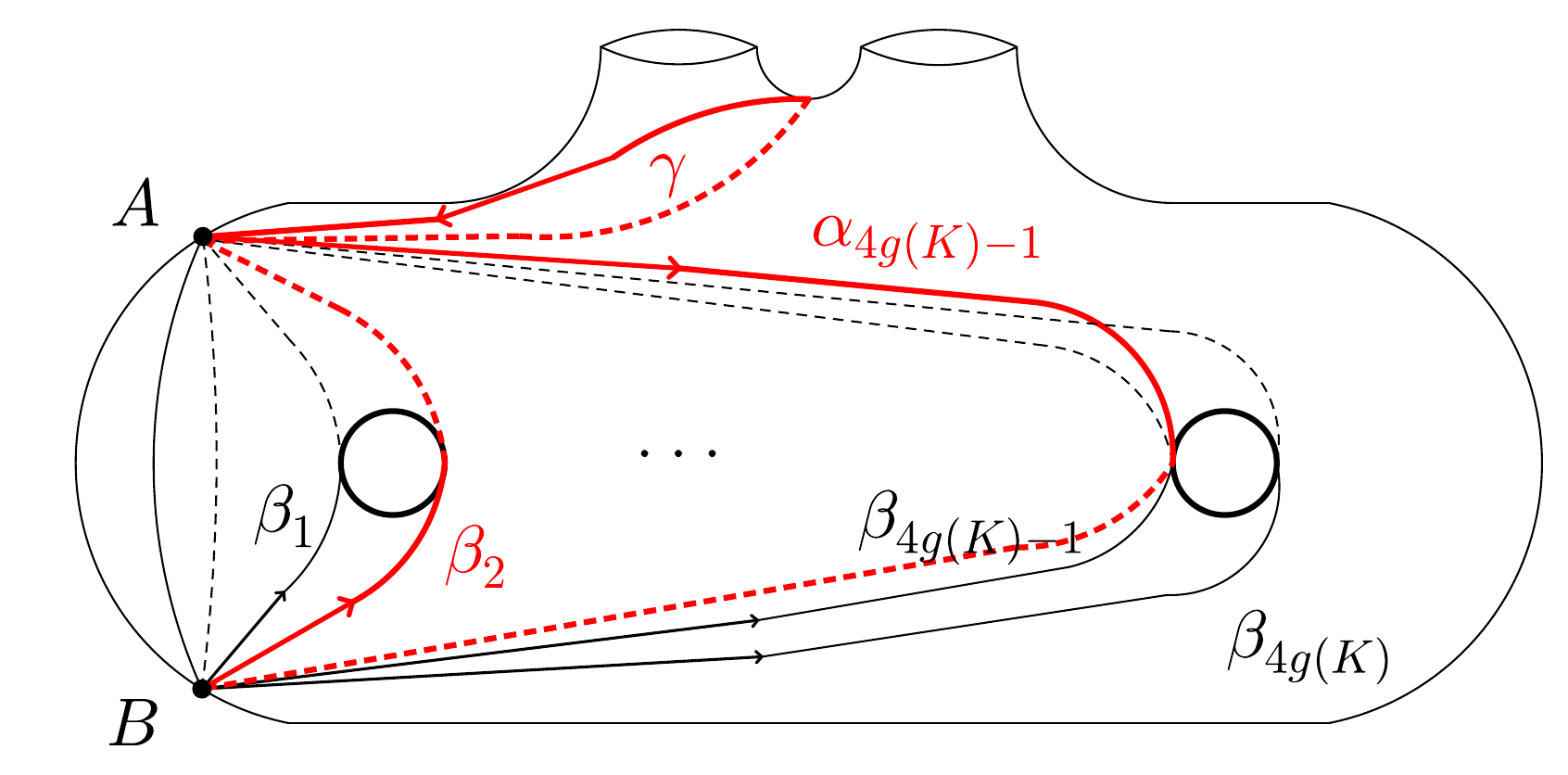}
\end{center}
\caption{Decomposition of  $B_i$ on $\Sigma_{2g(K) + n-1}$ (simple closed curve in red color is a decomposition of $B_2 \simeq \beta_2*\gamma*\alpha_{4g(K)-1}$) }\label{fig:B_i}
\end{figure}

\begin{remark}
Although the oriented paths $\alpha_0$ and $\widetilde{\alpha}_{4g(K)+1}$ in Figure~\ref{fig:Handle_fiber} are not cores of any $2$-dimensional $1$-handle, we adopt this notation by convention. We define $\beta_0 = \alpha_0^{-1}$ as the reversed path of $\alpha_0$.  
We consider $\beta_i$ and $\gamma$ as oriented paths  in Figure~\ref{fig:B_i}.
For simplicity, we assume that each path $\alpha_i$  flows from $A$ to $B$, each path $\beta_i$  flows from $B$ to $A$, and path $\gamma$ flows from $A$ to $A$. $\alpha^{-1}$ denotes the reversed path of $\alpha$.
\end{remark}

\begin{remark}
When we draw full Kirby diagram, we must carefully trace all attaching spheres of $4$-dimensional $2$-handles in $F_1 \times D^2$  together with framing for each $2$-handle. We can assume that attaching sphere is not twisted on any $4$-dimensional $1$-handle and all twisting information is located in $4$-dimensional $0$-handle. 
However,  to simply identify $1$-handle/$2$-handle canceling pairs, it suffices to consider the homotopy type of each vanishing cycle in 
$\pi_1(F_1\times D^2, (b_0, \mathbb{O}))$.  
\end{remark}

\subsection{\texorpdfstring{$E(1)_K$}{E(n)K} case}

\begin{lemma}
\label{Lemma:B_i}
In the monodromy factorization of $E(1)_K$ where $K$ is a fibered knot of genus $g(K)$,
 $B_i$ is homotopic to  $\beta_i * \alpha_{4g(K) + 1 -i} $ for each $i=1, 2, \cdots, 2g(K)$. Specifically:
\[
B_0 \simeq \beta_0*\widetilde{\alpha}_{4g(K) +1} \simeq  \alpha_0^{-1}*\alpha_{4g(K)}*\alpha_{4g(K)-1}^{-1} * \cdots * \alpha_{2} * \alpha_{1}^{-1}*\alpha_0
\]
and 
\[
B_i \simeq  \alpha_0^{-1}* \alpha_1 *  \cdots * \alpha_{i-1}^{(-1)^i}* \alpha_i^{(-1)^{i+1}}*  \alpha_{i-1}^{(-1)^i}* \cdots *\alpha_1*\alpha_0^{-1} *\alpha_{4g(K) + 1 -i}
\]
for each $i=1, 2,  \cdots, 2g(K)$.
\end{lemma}

\begin{proof}
In the case of $E(1)_K$, each $\gamma$ is homotopic to a point. 
Thus, the oriented loop $B_i$ is homotopic to $\beta_i*\alpha_{2g(K) + 1 -i}$ as shown in Figure~\ref{fig:B_i}. 

Each oriented path $\beta_i$ defines a word in the letters $\{\alpha_1, \alpha_2, \cdots, \alpha_{4g(K) + 2n -2}\}$ as follows:  Traversing the oriented path  $\beta_i$, if we cross $\alpha_{j}^*$ from left to right  (relative to the surface orientation), we append the letter $\alpha_j$.  
Conversely, crossing from right to left contributes $\alpha_j^{-1}$. 
Since $\beta_i$ flows from $B$ to $A$, and we require the word to be homotopic relative to the boundary, we append $\alpha_0^{-1}$ to the start or end if the first  or last letter is positive.

%

Specifically, $\beta_i$ meets $\alpha_1^*, \alpha_2^*, \cdots, \alpha_i^*$ with the sign of $\alpha_j$ being $(-1)^{j+1}$ ($j=1, 2, \cdots, i$), and then meets $\alpha_{i-1}^*, \alpha_{i-2}^*, \cdots, \alpha_1^*$ with sign of $\alpha_j$ is  $(-1)^{j+1}$ ($j=i-1, i-2, \cdots, 1$).
Therefore
\[
\beta_i \simeq \alpha_0^{-1} *\alpha_1*\cdots *\alpha_{i-1}^{(-1)^i}*\alpha_i^{(-1)^{i+1}}*\alpha_{i-1}^{(-1)^{i}}*\cdots *\alpha_1*\alpha_0^{-1}
\]
for each $i=1, 2, \cdots, 4g(K)$. 
Similarly, since $\widetilde{\alpha}_{4g(K)+1}$ meets $\alpha_{4g(K)}^*, \alpha_{4g(K)-1}^*, \cdots, \alpha_1^*$ with the sign $(-1)^{j}$ ($j=4g(K), 4g(K)-1, \cdots, 1$),
\[
\widetilde{\alpha}_{4g(K) +1} \simeq \alpha_{4g(K)}*\alpha_{4g(K)-1}^{-1}* \cdots * \alpha_1^{-1}*\alpha_0.
\]
These yield the stated formula.
\end{proof}

\begin{figure}[tbh]
\begin{center}
\includegraphics[scale=0.3]{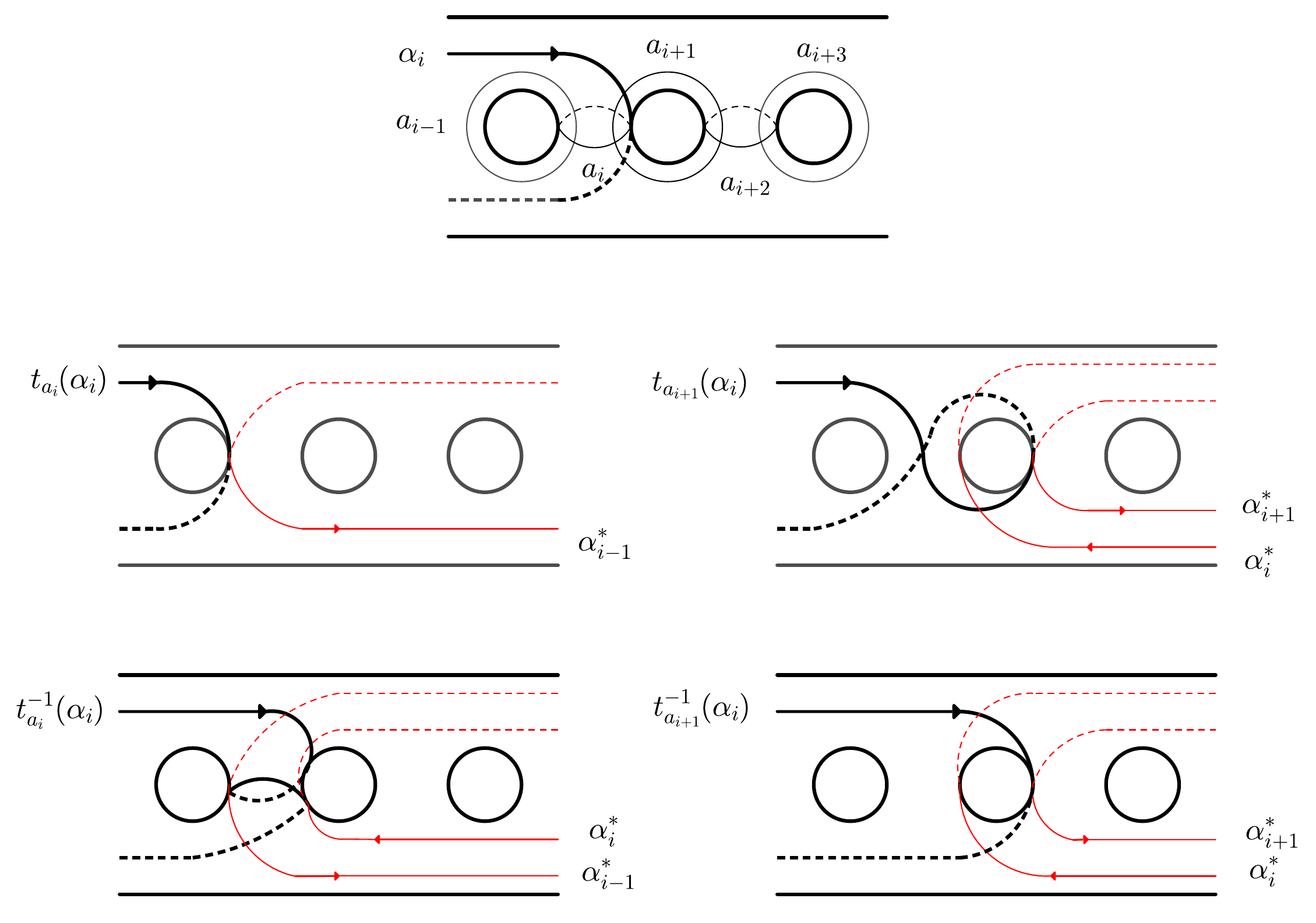}
\end{center}
\caption{word for $t_{a_j}^{\pm 1}(\alpha_i)$ }\label{fig:DehnTwist}
\end{figure}

\begin{lemma}
\label{Lemma:twist}
The following homotopy relations relative to the boundary hold:
\begin{itemize}
\item[(a)] $t_{a_{i+1}} (\alpha_i) \simeq \alpha_i * \alpha_{i+1}^{-1}* \alpha_i$ for each $i=0, 1, \cdots, 2g(K) -1$.
\item[(b)] $t_{a_{i+1}}^{-1}(\alpha_i) \simeq \alpha_{i+1}$ for each $i=0, 1, \cdots, 2g(K) -1$.
\item[(c)] $t_{a_i} (\alpha_i) \simeq \alpha_{i-1}$ for each $i=1, 2, \cdots, 2g(K)$.
\item[(d)] $t_{a_i}^{-1}(\alpha_i) \simeq \alpha_i * \alpha_{i-1}^{-1}*\alpha_i$ for each $i=1, 2, \cdots, 2g(K)$.
\item[(e)] if $j \ne i, i+1$, then $t_{a_j}(\alpha_i) = \alpha_i$.
\end{itemize}
\end{lemma}

\begin{proof}
This is evident from Figure~\ref{fig:DehnTwist}. For example, the word for $t_{a_{i+1}} (\alpha_i)$ corresponds to the path $\alpha_i * \alpha_{i+1}^{-1}* \alpha_i$, which is homotopic relative to the boundary.
\end{proof}

\begin{lemma}
\label{Lemma:alpha_i}
If  $\phi= t_{a_{1}}^{\epsilon_{1}} \circ t_{a_2}^{\epsilon_2} \circ \cdots  \circ t_{a_{2g(K)}}^{\epsilon_{2g(K)}}$ with $\epsilon_i \in \{ +1, -1\}$, then $\phi(\alpha_i)$ is homotopic to a word in 
$\{ \alpha_0^{\pm 1}, \alpha_1^{\pm 1}, \alpha_2^{\pm 1}, \cdots, \alpha_i^{\pm 1}, \alpha_{i+1}^{\pm 1} \}$ 
containing exactly one $\alpha_{i+1}$ or $\alpha_{i+1}^{-1}$ for each $i= 0, 1,  \cdots, 2g(K)-1$.
\end{lemma}

\begin{proof}
For $i=0$, 
\[
\phi(\alpha_0) = t_{a_1}^{\epsilon_1}(\alpha_0) \simeq 
\begin{cases}
\alpha_0* \alpha_1^{-1}*\alpha_0, & \epsilon_1 = +1, \\
\alpha_1, & \epsilon_1 = -1
\end{cases}
\]
by using Lemma~\ref{Lemma:twist}. 

We claim that $t_{a_1}^{\epsilon_1} \circ \cdots  \circ t_{a_k}^{\epsilon_k} (\alpha_k)$ is a word in  $\{\alpha_0^{\pm 1}, \alpha_1^{\pm 1}, \alpha_2^{\pm 1}, \cdots, \alpha_{k-1}^{\pm 1}, \alpha_{k}^{\pm 1} \}$ for any $k=1, \cdots, 2g(K)$.
This holds for $k=1$ because
\[
t_{a_1}^{\epsilon_1}(\alpha_1) \simeq 
\begin{cases} 
\alpha_0, & \epsilon_1 = 1 \\
\alpha_1*\alpha_0^{-1}*\alpha_1, & \epsilon_1 = -1.
\end{cases}
\]
Assume it holds for $k= 1, 2, \cdots, i-1$. 
Since
\begin{eqnarray*}
& & t_{a_1}^{\epsilon_1} \circ \cdots  \circ t_{a_i}^{\epsilon_i} (\alpha_i) \simeq t_{a_1}^{\epsilon_1} \circ \cdots  \circ t_{a_{i-1}}^{\epsilon_{i-1}} (t_{a_i}^{\epsilon_i} (\alpha_i)) \\
&\simeq& \begin{cases}
t_{a_1}^{\epsilon_1} \circ \cdots  \circ t_{a_{i-1}}^{\epsilon_{i-1}} (\alpha_{i-1}), & \epsilon_i = 1, \\
t_{a_1}^{\epsilon_1} \circ \cdots  \circ t_{a_{i-1}}^{\epsilon_{i-1}} (\alpha_i*\alpha_{i-1}^{-1}* \alpha_i) = \alpha_i*(t_{a_1}^{\epsilon_1} \circ \cdots  \circ t_{a_{i-1}} (\alpha_{i-1}^{-1}))*\alpha_i, &  \epsilon_i = -1 
\end{cases}
\end{eqnarray*}
by Lemma~\ref{Lemma:twist}, it holds for $k=i$.

For  $k= 1, 2, \cdots, 2g(K)-1$, we know that
\[
t_{a_{k+1}}^{\epsilon_{k+1}}(\alpha_k) \simeq 
\begin{cases}
\alpha_k*\alpha_{k+1}^{-1} *\alpha_k, & \epsilon_{k+1} = +1,\\
\alpha_{k+1}, & \epsilon_{k+1} = -1
\end{cases}
\]
and 
\begin{eqnarray*}
\phi(\alpha_k)  &\simeq&  t_{a_1}^{\epsilon_1} \circ \cdots  \circ t_{a_k}^{\epsilon_k} ( t_{a_{k+1}}^{\epsilon_{k+1}}(\alpha_k)) \\
&\simeq &
\begin{cases}
(t_{a_1}^{\epsilon_1} \cdots t_{a_k}^{\epsilon_k} (\alpha_k))*\alpha_{k+1}^{-1} *(t_{a_1}^{\epsilon_1} \cdots t_{a_k}^{\epsilon_k} (\alpha_k)), & \epsilon_{k+1} = +1,\\
\alpha_{k+1}, & \epsilon_{k+1} = -1.
\end{cases}
\end{eqnarray*}
By the above claim, 
$t_{a_1}^{\epsilon_1} \circ \cdots  \circ t_{a_k}^{\epsilon_k} (\alpha_k)$ is a word in $\{\alpha_0^{\pm 1}, \alpha_1^{\pm 1}, \alpha_2^{\pm 1}, \cdots, \alpha_{k-1}^{\pm 1}, \alpha_{k}^{\pm 1} \}$. 
Therefore $\phi(\alpha_k)$ is a word in $\{\alpha_0^{\pm 1}, \alpha_1^{\pm 1}, \alpha_2^{\pm 1}, \cdots, \alpha_{k}^{\pm 1}, \alpha_{k+1}^{\pm 1} \}$ containing exactly one $\alpha_{k+1}$ or  $\alpha_{k+1}^{-1}$.
\end{proof}

\begin{theorem}
\label{Theorem:E(1)_K}
For any fibered two-bridge knot $K$, $E(1)_K$ admits a handle decomposition without $1$- and $3$-handles.
\end{theorem}

\begin{proof}
We demonstrate that  $X_i$ ($i=1, 2$) has a handle decomposition without $1$-handles where $X_1$ is the Lefschetz fibration over $D^2$ with  monodromy factorization $\Phi_K(W) \cdot W$, and $X_2$  has monodromy factorization $W \cdot \Phi_K(W)$. 

Taking the mirror image of $K = D(\epsilon_1, \epsilon_2, \cdots, \epsilon_{2k})$, we have
\[
K^* = D(-\epsilon_1, -\epsilon_2, \cdots, -\epsilon_{2k})
\]
with monodromy 
$\phi_{K^*} = t_{a_{2k}}^{-\epsilon_{2k}} \circ \cdots \circ t_{a_2}^{-\epsilon_2} \circ t_{a_1}^{-\epsilon_1}$. 

Since 
$\Phi_K(W) \cdot W = \Phi_K(W \cdot \Phi_K^{-1}(W))$,
$X_1$ is isomorphic to the Lefschetz fibration with monodromy factorization $W \cdot \Phi_K^{-1}(W)$. 
Similarly,  $X_2$ is isomorphic to the Lefschetz fibration with monodromy factorization $\Phi_K^{-1}(W) \cdot W$.
Thus, each of $X_1$ and $X_2$ contains $(4g(K)+5)$ $2$-handles with attaching spheres given by the corresponding vanishing cycles:
\[
 B_0, B_1, \cdots, B_{2g(K)}, c_1, \phi(B_0), \phi(B_1), \cdots, \phi(B_{2g(K)}), \phi(c_1), \partial F_1
\]
where
 $\phi= t_{a_{1}}^{\epsilon_{1}} \circ t_{a_2}^{\epsilon_2} \circ \cdots  \circ t_{a_{2g(K)}}^{\epsilon_{2g(K)}}$ and $\epsilon_i \in \{ +1, -1\}$.

We decompose $\phi(B_0)$ as follows: 
\begin{eqnarray*}
\phi(B_0) & =& \phi( \beta_0 *\widetilde{ \alpha}_{4g(K) +1} ) = t_{a_1}^{\epsilon_1}(\alpha_0^{-1}) *  \alpha_{4g(K) +1}\\
&\simeq& \begin{cases} 
\alpha_0^{-1}*\alpha_1*\alpha_0^{-1}* \widetilde{ \alpha}_{4g(K) +1}, & \text{if } \epsilon_1 = 1, \\
\alpha_1^{-1} *  \widetilde{\alpha}_{4g(K) +1}, & \text{if } \epsilon_1 = -1.
\end{cases}
\end{eqnarray*}
Let $\widetilde{H}_1$ be the $2$-handle obatined by sliding $\phi(B_0)$ over $B_0$.
The attaching circle of $\widetilde{H}_1$ is homotopic to 
$\begin{cases} 
\alpha_0^{-1}*\alpha_1, & \epsilon_1 =1, \\ 
\alpha_1^{-1}*\alpha_0, & \epsilon_1 = -1. 
\end{cases} $ 
Therefore, the $2$-handle $\widetilde{H}_1$ and the $1$-handle $\alpha_1^*$ forms a canceling pair.

We have $B_1 \simeq \beta_1*\alpha_{4g(K)}$ and $\phi(B_1) \simeq \phi(\beta_1)* \alpha_{4g(K)}$. 
Sliding $\phi(B_1)$ over $B_1$ yields a $2$-handle $H_2$ whose attaching circle is homotopic to $\phi(\beta_1)* \beta_1^{-1}$. 
Since 
\begin{eqnarray*}
\phi(\beta_1) &=& \phi(\alpha_0^{-1}*\alpha_1*\alpha_0^{-1}) =(t_{a_1}^{\epsilon_1}(\alpha_0))^{-1}* (t_{a_1}^{\epsilon_1}t_{a_2}^{\epsilon_2}(\alpha_1))*(t_{a_1}^{\epsilon_1}(\alpha_0))^{-1}\\
&\simeq& \begin{cases} 
\alpha_0^{-1}*\alpha_1*\alpha_2^{-1}*\alpha_1*\alpha_0^{-1} , &  \epsilon_2= 1, \\
\alpha_0^{-1}*\alpha_1*\alpha_0^{-1}*\alpha_2*\alpha_0^{-1}*\alpha_1*\alpha_0^{-1}, & (\epsilon_1, \epsilon_2) = (1,-1),\\
\alpha_1^{-1}*\alpha_2*\alpha_1^{-1},& (\epsilon_1, \epsilon_2) = (-1,-1).
\end{cases}
\end{eqnarray*}
by Lemma~\ref{Lemma:twist}, 
\[
\phi(\beta_1)* \beta_1^{-1} \simeq
\begin{cases} 
\alpha_0^{-1}*\alpha_1*\alpha_2^{-1}*\alpha_0 , & \epsilon_2 = 1, \\
\alpha_0^{-1}*\alpha_1*\alpha_0^{-1}*\alpha_2, & (\epsilon_1, \epsilon_2) = (1,-1),\\
\alpha_1^{-1}*\alpha_2*\alpha_1^{-1}*\alpha_0*\alpha_1^{-1}*\alpha_0,& (\epsilon_1, \epsilon_2) = (-1,-1).
\end{cases}
\]
We slide $H_2$ over $\widetilde{H}_1$ whenever there is a letter $\alpha_1^{\pm 1}$, 
and then we obtain a $2$-handle $\widetilde{H}_2$ whose attaching circle goes over the $4$-dimensional $1$-handle $\alpha_2^*$ exactly one time. Therefore  $\widetilde{H}_2$ and $\alpha_2^*$ forms a canceling pair.

In general, 
sliding $\phi(B_{i-1})$ over $B_{i-1}$ yields a $2$-handle $H_i$ ($i= 1, 2, \cdots, 2g(K)$) whose attaching circle is homotopic to
\[
\phi(\alpha_1)*\phi(\alpha_2)^{-1}* \cdots * \phi(\alpha_{i-1})^{(-1)^i}*\cdots *\phi(\alpha_1)* \alpha_1^{-1}* \alpha_2* \cdots* \alpha_{i-1}^{(-1)^{i-1}}* \cdots *\alpha_2*\alpha_1^{-1}.
\]
Applying Lemma~\ref{Lemma:alpha_i}, 
it is a word of letters in $\{\alpha_0^{\pm 1}, \alpha_1^{\pm 1}, \alpha_2^{\pm 1}, \cdots, \alpha_{i-1}^{\pm 1}, \alpha_{i}^{\pm 1} \}$  
and the word contains exactly one $\alpha_{i}$ or $\alpha_{i}^{-1}$ . 
We slide $H_i$ over previous handles $\widetilde{H}_j$ ($j=1, 2, \cdots, i-1$) to isolate $\alpha_i^{\pm 1}$, 
and then obtain a $4$-dimensional $2$-handle  $\widetilde{H}_i$. 
 We observe that each $2$-handle $\widetilde{H}_i$ ($i=1, 2, \cdots, 2g(K)$) passes through the $4$-dimensional $1$-handle $h_i^1 \times D^2$ and  it does not pass through any other $1$-handles.
 So  $(\alpha_i^*, \widetilde{H}_i)$ forms a canceling pair for each $i=1, 2, \cdots, 2g(K)$.

We have
\[
B_i \simeq  \alpha_0^{-1}* \alpha_1 *  \cdots * \alpha_{i-1}^{(-1)^i}* \alpha_i^{(-1)^{i+1}}*  \alpha_{i-1}^{(-1)^i}* \cdots *\alpha_1*\alpha_0^{-1} *\alpha_{4g(K) + 1 -i}
\]
from Lemma~\ref{Lemma:B_i}. 
For each $i=1, 2, \cdots, 2g(K)$, 
let $\widetilde{H}_{2g(K)+i}$ be the $2$-handle obtained by sliding $B_{2g(K) + 1 -i}$ over $\widetilde{H}_j$ ($j=1, 2, \cdots,  2g(K) + 1 -i$) to isolate $\alpha_{2g(K)+i}$, 
then $(\alpha_{2g(K) + i}^*, \widetilde{H}_{2g(K)+i})$ forms a canceling pair for each $i=1, 2, \cdots, 2g(K)$.

Therefore, each $X_j$ has $4g(K)$ $1$- and $2$-handle canceling pairs,
\[
(\alpha_i^*, \widetilde{H}_i) \text{ and } ( \alpha_{2g(K)+i}^*, \widetilde{H}_{2g(K) + i})  \quad (i=1, 2, \cdots, 2g(K)),
\]
leaving one $0$-handle and five $2$-handles.  
The total handle structure of $E(1)_K$ thus comproses one $0$-handle, one $4$-handle and ten $2$-handles.
\end{proof}

\begin{remark}
We have extra five $4$-dimensional $2$-handles on each $X_i$ which are not used for the construction of canceling $2$-handles. Each of them can be modified by $2$-handle slides so that its attaching circle does not go through any $1$-handles.  So they do not cause any problem in the $1$-handle/$2$-handle canceling process.
\end{remark}

\subsection{\texorpdfstring{$E(n)_K$}{E(n)K} case for \texorpdfstring{$n\ge 2$}{nge2}}

The proof is essentially the same as in the case $n=1$, except for the decomposition of $B_i$ ($i=0, 1, \cdots, 4g(K)$) into subarcs and the construction of additional canceling $2$-handle for each $1$-handle $\alpha_{4g(K) + i}^*$ ($i=1, 2, \cdots, 2n-2$).  

\begin{theorem} 
\label{Theorem:EnK}
For any fibered two-bridge knot $K$ and any integer $n \ge 2$, the knot surgery $4$-manifod
$E(n)_K$ admits a handle decomposition without $1$- and $3$-handles.
\end{theorem}

\begin{proof}
Let $\alpha_i$, $\beta_i$ and $\gamma$ be the paths shown in Figure~\ref{fig:B_i}, 
then 
\[
B_i \simeq \beta_i * \gamma * \alpha_{4g(K) + 1 -i},  \qquad i=1, 2, \cdots, 2g(K)
\]
and 
\[
B_0 \simeq \beta_0 * \gamma* \widetilde{\alpha}_{4g(K) + 1}.
\]
Let $\phi = t_{a_1}^{\epsilon_1} \circ t_{a_2}^{\epsilon_2} \circ \cdots \circ t_{a_{2g(K)}}^{\epsilon_{2g(K)}}$ as defined previously, 
then
\begin{eqnarray*}
\phi(B_0) &\simeq& \phi(\beta_0 )* \gamma* \widetilde{\alpha}_{4g(K) + 1}, \\
\phi(B_i) &\simeq&  \phi(\beta_i)* \gamma*\alpha_{4g(K) + 1 -i} \quad (i=1, 2, \cdots, 2g(K)),
\end{eqnarray*}
because each $a_i$ is disjoint from $\gamma$,  $\widetilde{\alpha}_{4g(K) + 1}$ and $\alpha_{4g(K) + 1 -i}$.

Let $H_i$ be the $2$-handle obtained by sliding $\phi(B_{i-1})$ over $B_{i-1}$ for each $i=1, 2, \cdots, 2g(K)$. 
The attaching circle of $H_i$ is then homotopic to 
\[
\phi(\beta_{i-1}) * \beta_{i-1}^{-1}.
\]
As stated in Lemma~\ref{Lemma:B_i},
each $\beta_{i-1}$ is a word in $\{\alpha_0^{\pm 1}, \alpha_1^{\pm 1}, \cdots, \alpha_{i-1}^{\pm 1}\}$ 
containing exactly one instance of  $\alpha_{i-1}^{\pm 1}$.
Applying Lemma~\ref{Lemma:alpha_i} to each letter $\phi(\alpha_j)$ in the word (for $j=0, 1, \cdots, i-1$)
yields a word in $\{\alpha_0^{\pm 1}, \alpha_1^{\pm 1}, \cdots, \alpha_{i}^{\pm 1}\}$ 
containing exactly one instance of $\alpha_{i}^{\pm 1}$, as shown in the proof of Theorem~\ref{Theorem:E(1)_K}.

Let $\widetilde{H}_1 = H_1$.
For $i \ge 2$, let $\widetilde{H}_i$ ($i\ge 2$) be obtained from $H_i$ by sliding it over $\widetilde{H}_j$ (where $j<i$) corresponding to each letter $\alpha_j^{\pm 1}$ in the attaching circle's word.  
Then $\widetilde{H}_i$ becomes a canceling $2$-handle for the  $1$-handle associated with $\alpha_i$ ($i=1, 2, \cdots, 2g(K)$).

Using the cores and co-cores of $1$-handles as in Figure~\ref{fig:Handle_fiber}, 
we obtain
\begin{eqnarray*}
c_1 &\simeq& \alpha_{4g(K) + 2n -3} * \alpha_0^{-1}, \\
c_{2i} &\simeq& \alpha_{4g(K) + 2n -2i} * \alpha_{4g(K) + 2n -2i-1}^{-1}, \qquad  i=1, 2, \cdots, n-1, \\
c_{2i+1} &\simeq& \alpha_{4g(K)+2n -2i-3} * \alpha_{4g(K) +2n-2i-1}^{-1},  \qquad  i=1, 2, \cdots, n-2, \\
c_{2n-1} &\simeq&  \alpha_{2g(K)}^{-1}*\alpha_{2g(K)-1} * \cdots*\alpha_1 * \alpha_{4g(K)+1}^{-1}*\alpha_{2g(K)}*\alpha_{2g(K)-1}^{-1} * \cdots*\alpha_1^{-1}*\alpha_0 
\end{eqnarray*}
by selecting an orientation of each $c_i$.

We define the canceling pairs as follows:
\begin{itemize}
\item[(1)] Let $\widetilde{H}_{4g(K) + 2n -3}$ be the $2$-handle derived from $t_{c_1}$;  
it is a canceling pair for the $1$-handle $\alpha_{4g(K) + 2n -3}^*$. 
\item[(2)] Let $H_{4g(K) + 2n -2}$ be the $2$-handle derived from $t_{c_2}$. Sliding it over  $\widetilde{H}_{4g(K) + 2n -3}$ yields $\widetilde{H}_{4g(K) + 2n -2}$ which cancels the $1$-handle $\alpha_{4g(K) + 2n -2}^*$. 
\item[(3)]  Generally, 
 define $H_{4g(K) + 2n -2i-3}$  as the $2$-handle derived from $t_{c_{2i+1}}$. 
 Sliding it over  $\widetilde{H}_{4g(K) + 2n -2i -1}$ yields $\widetilde{H}_{4g(K) + 2n -2i-3}$ 
 which cancels the $1$-handle $\alpha_{4g(K) + 2n -2i-3}^*$ ($i= 1, 2,  \cdots, n-2$).
\item[(4)] Similarly, define $H_{4g(K) + 2n -2i}$ as the $2$-handle derived from $t_{c_{2i}}$. 
Sliding it over  $\widetilde{H}_{4g(K) + 2n -2i -1}$ yields $\widetilde{H}_{4g(K) + 2n -2i}$, 
 which cancels the $1$-handle $\alpha_{4g(K) + 2n -2i}^*$ ($i= 1, 2, \cdots, n-1$). 
 \end{itemize}
Thus,  $( \alpha_{4g(K)+i}^*,  \widetilde{H}_{4g(K) + i})$ forms $1$-handle/$2$-handle canceling pairs for each $i =1, 2, \cdots, 2n-2$.

It remains to find a canceling $2$-handle for each $\alpha_{2g(K) + i}$  ($i=1, 2, \cdots, 2g(K)$). 
From the cores and co-cores of $1$-handles as in Figure~\ref{fig:Handle_fiber}, 
we verify that 
\begin{eqnarray*}
B_1 &\simeq& \alpha_1*\alpha_{4g(K) + 1}^{-1} * \alpha_{4g(K)}*\alpha_0^{-1}, \\
B_2 &\simeq& \alpha_1*\alpha_2^{-1}*\alpha_1*\alpha_{4g(K)+1}^{-1}*\alpha_{4g(K)-1}*\alpha_0^{-1}, \\
B_i &\simeq& \alpha_1*\cdots * \alpha_{i-1}^{(-1)^i}*\alpha_i^{(-1)^{i+1}}*\alpha_{i-1}^{(-1)^i}*\cdots *\alpha_1*\alpha_{4g(K)+1}^{-1}*\alpha_{4g(K)+1 -i}*\alpha_0^{-1}, \\
&&  (i=3, 4, \cdots, 2g(K)).
\end{eqnarray*}
We slide $B_{2g(K) + 1-i}$ over $\widetilde{H}_j$ ($j= 1, 2, \cdots, 2g(K)+1 -i, 4g(K) +1$) 
whenever the word for $B_{2g(K) + 1-i}$ contains the letter $\alpha_j$. 
We denote the resulting $2$-handle as $\widetilde{H}_{2g(K) + i}$ ($i=1, 2, \cdots, 2g(K)$), then $\widetilde{H}_{2g(K) + i}$ cancels the $1$-handle $\alpha_{2g(K) + i}^*$.
\end{proof}


\section{ Handle decomposition of $E(n)_K$ for a family of Stallings knot} 

In this section, we consider the knot surgery $4$-manifold $E(n)_{K_m}$  and its handle decomposition, where $K_m$ is the family of Stallings knot ~\cite{Stallings:78} obtained from $3_1\sharp 3_1^*$ by performing Stallings twist $m$ times ($m \in \mathbb{Z}$) where $3_1$ is the left-handed trefoil knot. It is known that $K_m$ is a $3$-bridge knot.

\begin{theorem}
\label{theorem:Stallings}
For any positive integer $n$ and any integer $m \in \mathbb{Z}$, the knot surgery $4$-manifold $E(n)_{K_m}$ admits a  handle decomposition without $1$- and $3$-handles.
\end{theorem}

\begin{proof}
It is well known that $K_m$ is a genus $2$ fibered knot and it has monodromy
\[
\phi_m = t_{a_3}^m \circ  t_{a_4}  \circ t_{b_2} \circ t_{a_2}^{-1} \circ t_{a_1}^{-1} 
\]
where $a_i$ ($i=1, 2, 3, 4$) and $b_2$ are shown in Figure~\ref{fig:Curves}. 

As mentioned in subsection~\ref{subsection:MF}, $E(n)_{K_m}$  decomposed into two Lefschetz fibrations $X_i$ ($i=1, 2$) over $D^2$ with monodromy factorizations $\Phi_{K_m}(W) \cdot W$ and $W \cdot \Phi_{K_m}(W)$. 
We observe that $\Phi_{K_m} = \phi_m = t_{a_3}^m \circ \phi_0$.
\medskip

\begin{figure}[htb]
\begin{center}
\includegraphics[scale=0.35]{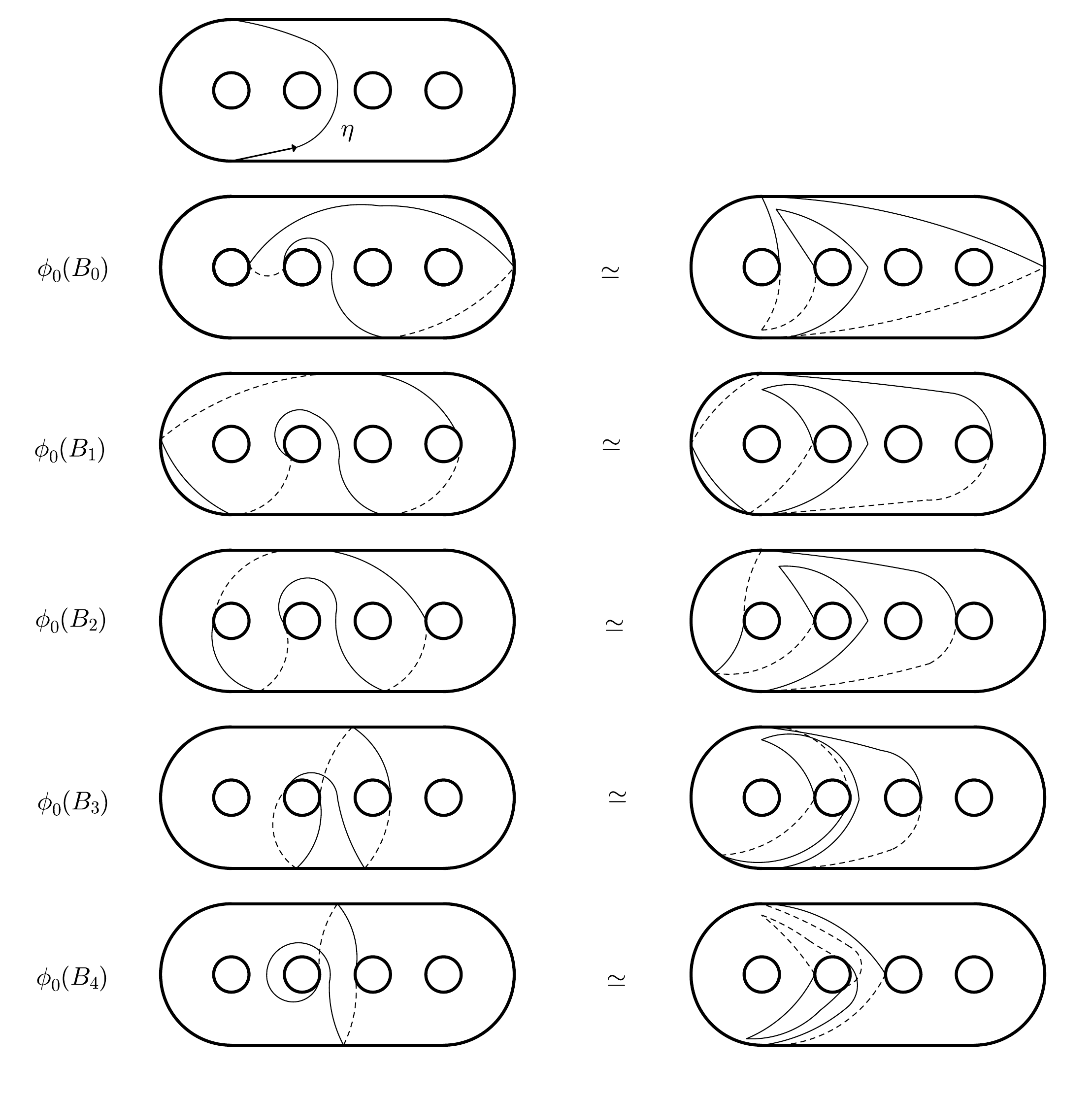}
\end{center}
\caption{word for $\phi_0(B_i)$ }\label{fig:Phi0}
\end{figure}

\noindent
{\bf Case $n=1$: }
$X_1$ contains $13$ $2$-handles 
\[
\phi_m(B_i),   \phi_m(c_1),  B_i, c_1, \partial F_1 \quad (i=0, 1, 2, 3, 4),
\] 
attached from right to left. Similarly,  $X_2$ also contains $13$ $2$-handles 
\[
B_{4-i},   c_1, \phi_m(B_{4-i}), \phi_m(c_1),  \partial F_1 \quad (i=0, 1, 2, 3, 4).
\] 
Let $\eta$ be an oriented path as in Figure~\ref{fig:Phi0}.
Then $ \alpha_0  * \eta \simeq \alpha_1*\alpha_2^{-1}*\alpha_3*\alpha_4^{-1}$.  Using homotopy, we find:
\begin{eqnarray*}
\phi_m(B_0) &\simeq& \eta*t_{a_3}^m(\alpha_3)*t_{a_3}^m(\alpha_2^{-1})*\widetilde{\alpha}_9, \\
\phi_m(B_1) &\simeq&\eta*t_{a_3}^m(\alpha_3)* \beta_0*\alpha_8, \\
\phi_m(B_2) &\simeq& \eta*t_{a_3}^m(\alpha_3)*\beta_1*\alpha_7 , \\
\phi_m(B_3) &\simeq& \eta*t_{a_3}^m(\alpha_3)*\beta_4*\alpha_6,\\
\phi_m(B_4) &\simeq&  \beta_4*t_{a_3}^m(\beta_3^{-1})*\beta_4*\alpha_5.
\end{eqnarray*}
This holds because (as in figure~\ref{fig:Phi0}):
\begin{eqnarray*}
\phi_0(B_0) &\simeq& \eta*\alpha_3*\alpha_2^{-1}*\widetilde{\alpha}_9, \\
\phi_0(B_1) &\simeq&\eta*\alpha_3* \beta_0*\alpha_8, \\
\phi_0(B_2) &\simeq& \eta* \alpha_3*\beta_1*\alpha_7 , \\
\phi_0(B_3) &\simeq& \eta*\alpha_3*\beta_4*\alpha_6,\\
\phi_0(B_4) &\simeq&  \beta_4*\beta_3^{-1}*\beta_4*\alpha_5.
\end{eqnarray*}
and the simple closed curve $a_3$ meets only  $\alpha_2^{-1}$, $\alpha_3$ and  $\beta_3^{-1}$.

We also have 
\begin{eqnarray*}
B_0 &\simeq& \beta_0 * \widetilde{\alpha}_9, \\
B_1 &\simeq& \beta_1*\alpha_8,\\
B_2 &\simeq& \beta_2* \alpha_7,\\
B_3 &\simeq& \beta_3*\alpha_6,\\
B_4 &\simeq& \beta_4*\alpha_5.
\end{eqnarray*}

Let $H_i$ be the $4$-dimensional $2$-handle obtained by sliding $\phi_m(B_i)$ over $B_i$ ($i=0, 1, 2, 3, 4$). The attaching circles are homotopic to:
\begin{eqnarray*}
H_0 &:& \eta*t_{a_3}^m(\alpha_3)*t_{a_3}^m(\alpha_2^{-1})*\beta_0^{-1} \simeq \eta*\alpha_3*\alpha_2^{-1}*\alpha_0, \\
H_1 &:& \eta*t_{a_3}^m(\alpha_3)*\beta_0*\beta_1^{-1}, \\
H_2 &:& \eta*t_{a_3}^m(\alpha_3)*\beta_1*\beta_2^{-1} \simeq  \eta*t_{a_3}^m(\alpha_3)*\alpha_0^{-1}*\alpha_2*\alpha_1^{-1}*\alpha_0 , \\
H_3 &:& \eta*t_{a_3}^m(\alpha_3)*\beta_4*\beta_3^{-1}, \\
H_4 &:& \beta_4*t_{a_3}^m(\beta_3^{-1}).
\end{eqnarray*}

Let $\widehat{H}_{i,j}$ be the $2$-handle obtained by sliding $H_i$ over $H_j$ . Then attaching circles are homotopic to:
\begin{eqnarray*}
\widehat{H}_{1,2} &:&  \beta_2*\beta_1^{-1}*\beta_0*\beta_1^{-1} \simeq \alpha_2^{-1}*\alpha_0,\\
\widehat{H}_{3, 2} &:& \beta_4 *\beta_3^{-1}*\beta_2*\beta_1^{-1} \simeq \alpha_0^{-1}*\alpha_1*\alpha_2^{-1}*\alpha_3*\alpha_4^{-1} *\alpha_0 \simeq \eta * \alpha_0.
\end{eqnarray*}

We perform the $2$-handle slides in the following order:
\begin{itemize}
\item[(1)] Let $\widetilde{H}_2 = \widehat{H}_{1, 2}$. Then $(\alpha_2^*, \widetilde{H}_2 )$ forms a canceling pair.

\item[(2)] Let $\widetilde{H}_3$ be obtained by sliding $H_0$ over $\widehat{H}_{3, 2}$ and $\widetilde{H}_2$. Then $(\alpha_3^*, \widetilde{H}_3 )$  forms a canceling pair.

\item[(3)] Let $\widetilde{H}_1$ be obtained by sliding $H_2$ over $\widehat{H}_{3,2}$ and then over $\widetilde{H}_2$ or $\widetilde{H}_3$ corresponding to the letters $\alpha_2^{\pm 1}$ and $\alpha_3^{\pm 1}$, respectively, in the word $\alpha_0^{-1}*t_{a_3}^m(\alpha_3)\allowbreak *\alpha_0^{-1}*\alpha_2$. 
Then $(\alpha_1^*, \widetilde{H}_1 )$ forms a canceling pair.
(Recall that $t_{a_3}^m(\alpha_3)$ is a word in $\{\alpha_2^{\pm 1}, \alpha_3^{\pm 1}\}$ because $t_{a_3}(\alpha_3) \simeq \alpha_2$, $t_{a_3}(\alpha_2) \simeq \alpha_2*\alpha_3^{-1}*\alpha_2$, $t_{a_3}^{-1}(\alpha_3)\simeq \alpha_3*\alpha_2^{-1}*\alpha_3$ and $t_{\alpha_3}^{-1}(\alpha_2) \simeq \alpha_3$ by Lemma~\ref{Lemma:twist}.)

\item[(4)] Let $\widetilde{H}_4$ be obtained by sliding $H_3$ over $H_2$ (which is $\widehat{H}_{3, 2}$) and then over $\widetilde{H}_1$, $\widetilde{H}_2$ and $\widetilde{H}_3$. Then $(\alpha_4^*, \widetilde{H}_4 )$  forms a canceling pair.
\end{itemize}

In order to obtain the $2$-handle $\widetilde{H}_{4+i}$ ($i=1, 2, 3, 4$) which cancels the $1$-handle $\alpha_{4+i}^*$, we 
slide $B_{5-i}$ over $\widetilde{H}_j$ ($j= 1, 2, 3, 4$) whenever the letter $\alpha_j^{\pm 1}$ appears in the word for $B_{5-i}$.
\medskip

\noindent
{\bf  Case $n \ge 2$: } The process for finding $1$- and $2$-handle canceling pairs is  identical to Theorem~\ref{Theorem:EnK}. We first find canceling $2$-handle of $\alpha_i^*$ ($i=1, 2, 3, 4$) as in the $n=1$ case above. Then,  we find canceling $2$-handle of $\alpha_{8+i}^*$ ($i=1, 2, \cdots, 2n-2$) as in the proof of Theorem~\ref{Theorem:EnK}. Finally, we find canceling $2$-handle of $\alpha_i^*$ ($i=5, 6, 7, 8$) using $B_{5-i}$.
\end{proof}


%

\begin{thebibliography}{10}

\bibitem{Akbulut:09}
S.~Akbulut.
\newblock An infinite family of exotic {D}olgachev surfaces without 1- and
  3-handles.
\newblock {\em J. G\"okova Geom. Topol. GGT}, 3:22--43, 2009.

\bibitem{Akbulut:12}
S.~Akbulut.
\newblock The {D}olgachev surface. {D}isproving the {H}arer-{K}as-{K}irby
  conjecture.
\newblock {\em Comment. Math. Helv.}, 87(1):187--241, 2012.
\newblock \href {https://doi.org/10.4171/CMH/252} {\path{doi:10.4171/CMH/252}}.

\bibitem{Baykur:25}
R. I.~Baykur.
\newblock On four-manifolds without 1- and 3-handles.
\newblock {\em Proc. Amer. Math. Soc.}, 153(6):2681--2685, 2025.
\newblock \href {https://doi.org/10.1090/proc/17215}
  {\path{doi:10.1090/proc/17215}}.

\bibitem{BZ:03}
G.~Burde and H.~Zieschang.
\newblock {\em Knots}, volume~5 of {\em de Gruyter Studies in Mathematics}.
\newblock Walter de Gruyter \& Co., Berlin, second edition, 2003.

\bibitem{FS:2004}
R.~Fintushel and R.~Stern.
\newblock Families of simply connected 4-manifolds with the same
  {S}eiberg-{W}itten invariants.
\newblock {\em Topology}, 43(6):1449--1467, 2004.

\bibitem{Gurtas:2004}
Y.~Gurtas.
\newblock Positive {D}ehn twist expressions for some new involutions in mapping
  class group, 2004.
\newblock arXiv:math.GT/0404310.

\bibitem{Harer:82}
J.~Harer.
\newblock How to construct all fibered knots and links.
\newblock {\em Topology}, 21(3):263--280, 1982.

\bibitem{Kawauchi:1996}
A.~Kawauchi.
\newblock {\em A survey of knot theory}.
\newblock Birkh\"auser Verlag, Basel, 1996.
\newblock Translated and revised from the 1990 Japanese original by the author.

\bibitem{Matsumoto:96}
Y.~Matsumoto.
\newblock Lefschetz fibrations of genus two---a topological approach.
\newblock In {\em Topology and Teichm\"uller spaces (Katinkulta, 1995)}, pages
  123--148. World Sci. Publishing, River Edge, NJ, 1996.

\bibitem{Monden:25}
N.~Monden and R.~Yabuguchi.
\newblock Knot surgered elliptic surfaces for a $(2,2h+1)$-torus knot, 2025.
\newblock URL: \url{https://arxiv.org/abs/2503.06102}, \href
  {https://arxiv.org/abs/arXiv:2503.06102} {\path{arXiv:arXiv:2503.06102}}.

\bibitem{Stallings:78}
J.~Stallings.
\newblock Constructions of fibred knots and links.
\newblock In {\em Algebraic and geometric topology (Proc. Sympos. Pure Math.,
  Stanford Univ., Stanford, Calif., 1976), Part 2}, Proc. Sympos. Pure Math.,
  XXXII, pages 55--60. Amer. Math. Soc., Providence, R.I., 1978.

\bibitem{Yun:2006}
K.-H. Yun.
\newblock On the signature of a {L}efschetz fibration coming from an
  involution.
\newblock {\em Topology Appl.}, 153(12):1994--2012, 2006.

\bibitem{Yun:2008}
K.-H. Yun.
\newblock Twisted fiber sums of {F}intushel-{S}tern's knot surgery 4-manifolds.
\newblock {\em Trans. Amer. Math. Soc.}, 360(11):5853--5868, 2008.

\end{thebibliography}
%

\end{document}